\documentclass[12pt]{article}
\pdfoutput=1

\usepackage{amssymb}
\usepackage{amsthm}

\usepackage[a4paper, total={17.5cm, 27cm}]{geometry}
\usepackage{makeidx}         
\usepackage{graphicx}        
\usepackage{multicol}        
\usepackage[bottom]{footmisc}

\usepackage{newtxtext}       %
\usepackage{newtxmath}       
\usepackage{multirow}
\usepackage{diagbox}

\usepackage{soul,xcolor}

\usepackage[colorlinks]{hyperref}

\usepackage{booktabs}

\definecolor{darkgreen}{rgb}{0.0,0.4,0.0}
\definecolor{darkred}{rgb}{0.6,0.0,0.0}
\definecolor{darkblue}{rgb}{0.0,0.0,0.5}
\definecolor{gray}{rgb}{0.5,0.5,0.5}
\definecolor{cyan}{rgb}{0.0,1.0,1.0}
\definecolor{darkcyan}{rgb}{0.0,0.5,0.5}
\definecolor{darkorange}{rgb}{0.8,0.4,0.0}
\definecolor{darkmargenta}{rgb}{0.5,0.0,0.5}
\definecolor{black}{rgb}{0.0,0.0,0.0}

\hypersetup{citecolor=blue, linkcolor=red, anchorcolor=darkblue,
	filecolor=darkblue, urlcolor=darkblue}
\hypersetup{pdfauthor={Jakob Seierstad Stokke},pdftitle={Stokke}}

\def\red#1{{\color{red}#1}}

\setstcolor{red}

\usepackage{float} 
\usepackage{bm}
\usepackage{subcaption}
\usepackage{tcolorbox}
\usepackage{cleveref}

\makeindex             
\def \a  {\alpha}

\def \e  {\varepsilon}

\def \om {\omega}
\def \Om {\Omega}

\def \calA {\mathcal{A}}

\def \del {\nabla}
\def \div {\nabla\cdot}

\def \p  {\partial}
\def \R  {\mathbb{R}}
\def \N  {\mathbb{N}}

\def \bmu {\bm{u}}
\def \bmv {\bm{v}}

\def \bmF {\bm{f}}
\def \bms{\bm{\sigma}}

\def \tdiv {\normalfont\mbox{div}}
\def \tDiv {\normalfont\mbox{Div}}
\def \tgrad {\normalfont\mbox{grad}}
\def \tGrad {\normalfont\mbox{Grad}}

\def \i {\normalfont\red{\mbox{i}}}

\newtheorem{theorem}{Theorem}[section]

\newtheorem{lemma}[theorem]{Lemma}
\newtheorem{definition}[theorem]{Definition}
\newtheorem{problem}{Problem}
\newtheorem{remark}{Remark}

\usepackage{authblk}
\title{The Biot--Allard poro-elasticity system: equivalent forms and well-posedness}
        \author[1]{Jakob S. Stokke}
		
		\author[2]{Markus Bause}

        \author[2]{Nils Margenberg}
		
		\author[1]{Florin A. Radu}

\affil[1]{Center for Modeling of Coupled Subsurface Dynamics, Department of Mathematics, University of Bergen, Bergen, Norway}
		
\affil[2]{Helmut Schmidt University, Faculty of Mechanical and Civil Engineering, Hamburg, Germany}
\date{}
\begin{document}
	
		
		
		
	\maketitle


		
		
		

		\begin{abstract}
		We consider the fully dynamic Biot--Allard model, which includes memory effects. Convolution integrals in time model the history of the porous medium. We use a series representation of the dynamic permeability in the frequency domain to rewrite the equations in a coupled system without convolution integrals, suitable for the design of efficient numerical approximation schemes. The main result is the well-posedness of the system, proved by the abstract theory of R. Picard for evolutionary problems.
		\end{abstract}
		
		
		
		
		
		
		
	
	
	\section{Introduction}
	\label{sec:intro}
	
	In this letter, we consider poro-elasticty with memory effects, referred to as the Biot--Allard model, which models the coupling of flow and deformation in a fully saturated porous medium. Many relevant problems in natural sciences and engineering are described by poro-elasticity systems, including applications in geophysics, bio-mechanics, and bio-medicine. For example, seismic monitoring of the fluid flow during $CO_2$ injection in the subsurface is important for safe storage. 
 In this case, the memory effects in the Biot--Allard model could be interpreted as a delay effect due to drag forces. Experimental results in \cite{batzle2001} indicate that memory effects due to drag forces are relevant for rocks in seismic wave propagation.
    In bone remodeling, the memory effects due to fluid flow along cell bodies play a crucial role in activating bone cells \cite{robling2006}.

	In this work, we consider the fully dynamic Biot--Allard model from \cite{MikelicBiotAllard}: Find the solid phase displacement $(\bmu)$ and the pressure $(p)$ such that there holds	
	\begin{subequations}\label{eqs:BiotAllardconv}
		\begin{align}
			\rho\p_{t}^2\bmu-\div\bm{C}\e(\bmu)+\a\del p+\p_t\left(\calA\ast (\rho_f\bm{f}-\del p-\rho_f\p_t^2\bmu)\right)&=\rho\bm{f}\label{eq:biotallardmechanincs} \quad\mbox{ in }\Om\times(0,T], \\	
			c_0\p_tp+\div (\alpha\p_t\bmu)+\div \left(\calA\ast(\bmF-\frac{1}{\rho_f}\del p-\p_t^2\bmu)\right)&=0\label{eq:biotallardflow} \quad\quad\mbox{ in }\Om\times(0,T],
		\end{align}
		 together with initial and boundary conditions
		\begin{align}\label{eq:conditions}
			\bmu(0)=\bmu_0,\quad\p_t\bmu(0)=\bmu_1, \quad p(0)&=p_0, \quad\mbox{ in }\Om,\quad
			\bmu=\bm{0}, \quad p=0,\quad\mbox{ on }\p\Om\times(0,T].
		\end{align}
		In \eqref{eqs:BiotAllardconv}, the operator $\ast$ denotes the time convolution, i.e.\ $f\ast g (t) =\int_{0}^{t}f(t-s)g(s)ds$.  Further, $\Om$ is a bounded domain in $\R^{d}$ with $d\in\{2,3\}$ with boundary $\p\Om$, $T$ is the final time, $\alpha$ is the Biot coefficient, $\rho=\rho_f\varphi+\rho_s(1-\varphi)$ is the mass density with $\rho_s$ and $\rho_f$ being the solid and fluid density respectively and $\varphi$ the porosity. Furthermore, $\textbf{C}$ is Gassman's fourth order tensor, $c_0$ is the specific storage coefficient, and $\e(\bmu) :=\frac12(\del\bmu+\del\bmu^{T})$ is the linearized strain tensor. $\calA(\cdot)$ is the dynamic permeability tensor, accounting for the delay memory effects of the interaction between the pore walls and fluid. For a general definition of the dynamic permeability we refer to \cite{johnson1987}, but we note that the low-frequency limit of $\hat{\calA}(\om)$ should equal the static case or a similar condition in the time domain \cite[Eq.\ (31)]{MikelicBiotAllard}.  The tensor $\textbf{C}$ is assumed to be symmetric and positive definite and independent of space and time variables. The right-hand side function $\bmF$ in \eqref{eq:biotallardmechanincs} is an external force density. In \eqref{eq:conditions}, homogenous Dirichlet boundary conditions are prescribed for the sake of brevity and simplicity.
	\end{subequations}
	
 By an approximation of the dynamic permeability, we derive a set of equations without convolution terms, in which memory effects are represented by auxiliary differential equations (ADEs). In the future, this new system will allow for efficient and less memory-demanding numerical approximation schemes to the Biot--Allard model.
		Further, we prove the well-posedness of the thus obtained Biot--Allard model by rewriting it as a first-order evolutionary problem, i.e., a system of the form
		\begin{align}\label{eq:abstractproblem}
			(\p_t M(\p_t)+A)\bm{U}=\bm{G},
		\end{align} 
  where $M$ is the respective material law, cf.\ Def.\ \ref{def:33}.
   This allows us to apply the abstract theory developed by R. Picard \cite{picard2009,seifert2022} to prove the well-posedness of system \eqref{eq:abstractproblem} considered on the whole time axis. Eq.\ \eqref{eq:abstractproblem} is equipped with initial values by a generalization of the solution theory to distribution right-hand sides \cite[Chapter 9]{seifert2022}, which is not explicitly done here for brevity.
 For an example of the evolutionary theory applied to a hyperbolic or parabolic problem with memory effects, we refer to \cite{trostorff2015}.
 The regularity, existence, and uniqueness of Biot models were first studied by Showalter \cite{showalter2000}.
	
	The paper is structured in the following way. In Sec.\ \ref{sec:ADE}, we introduce an auxiliary differential equation for the Biot--Allard model, which allows us to remove the convolution integral in \eqref{eqs:BiotAllardconv}. In Sec.\ \ref{sec:evo}, a brief introduction to the solution theory of R. Picard for evolutionary systems is presented, and the Biot--Allard model written as an evolutionary problem is shown to be well-posed. In Sec.\ \ref{sec: general}, we derive the general ADE formulation, which relies on a higher-order approximation of the dynamic permeability, compared with Sec.\ \ref{sec:ADE}. 
 The letter ends with a concluding section.
	
\section{Biot--Allard system with an auxiliary differential equation}\label{sec:ADE}

Computing a numerical solution of \eqref{eqs:BiotAllardconv} is expensive and challenging. Therefore, we derive an ADE to avoid the expensive evaluation of the convolution integral in \eqref{eqs:BiotAllardconv}. To do this, we use a known result \cite{yvonne2014} that for any pore geometry the dynamic permeability in the frequency domain is approximated by 
\begin{align}\label{relation:dynamicpermeability}
\hat{\calA}(\om) = \frac{\eta_k}{F}\sum_{j=1}^{N}\frac{d_j}{1+\i\om c_j},
\end{align}
where $\i$ is the imaginary unit, $\eta_k=\eta/\rho_f$ is the kinetic viscosity, $F=\a_{\infty}/\varphi$ is the formation factor, $\a_{\infty}$ is the inﬁnite-frequency tortuosity and $c_j, d_j>0$ are constants, where $c_j<\Phi$ with $\Phi$ being the principal viscous relaxation time. For a specific frequency range, chosen for a particular application, the constants $c_j, d_j$ can be estimated from $N$ sample points in a procedure detailed in \cite{yvonne2014}. 
For simplicity in the derivation, we consider the case when $N=1$; for the general case see Sec.\ \ref{sec: general}. We note that \eqref{relation:dynamicpermeability} implies an exponential decay in time of $\calA(t)$, typically found in visco-elastic problems with memory effects. We refer to \cite{grafakos2014} for a detailed introduction to Fourier analysis. For our purpose, we introduce an auxiliary variable $\hat{\Psi}_1$ in the frequency domain, defined by 
	\begin{align}\label{aux1}
	\hat{\Psi}_1 :=\frac{\eta_k}{F}\frac{d_1}{1+\i\om c_1}\left(\hat{\bmF}-\frac{1}{\rho_f}\del \hat{p}+\om^{2}\hat{\bmu}\right).
\end{align}
	Then, the convolution in \eqref{eq:biotallardflow} can be expressed in the frequency domain by
	\begin{align}
	\calA\ast \left(\bmF-\frac{1}{\rho_f}\del p-\p_t^{2}\bmu\right)=\hat{\calA}\cdot\left(\hat{\bmF}-\frac{1}{\rho_f}\del \hat{p}+\om^{2}\hat{\bmu}\right)=\hat{\Psi}_1.
	\end{align}
	We also have an ADE for the auxiliary variable $\hat{\Psi}_1$ in the frequency domain
	\begin{align}
		\hat{\Psi}_1+\i\om c_1\hat{\Psi}_1 -\frac{d_1\eta_k}{F}\left(\hat{\bmF}-\frac{1}{\rho_f}\del \hat{p}+\om^{2}\hat{\bmu}\right) = 0.
	\end{align}
		By computing the inverse Fourier transform and reordering the equation, we obtain the ADE
	\begin{align}\label{cor}
		c_1\p_t\Psi_1 +\Psi_1+\frac{d_1\eta_k}{F}\frac{1}{\rho_f}\del p+\frac{d_1\eta_k}{F}\p_t^{2}\bmu=\frac{d_1\eta_k}{F}\bmF.
	\end{align}
	This results in the coupled system of equations in the time domain 
	\begin{subequations}\label{eqs:biotallardADE}
		\begin{align}
				\rho\p_{t}^2\bmu-\div\bm{C}\e(\bmu)+\a\del p +\rho_f\p_t \Psi_1&=\rho\bmF,\\
			c_0\p_tp+\div (\alpha\p_t\bmu)+\div \Psi_1&=0,\\
				c_1\rho_f\p_t\Psi_1 +\rho_f\Psi_1+\frac{d_1\eta_k}{F}\del p+\frac{d_1\eta_k}{F}\rho_f\p_t^{2}\bmu&=\frac{d_1\eta_k}{F}\rho_f\bmF.
		\end{align}
	\end{subequations}
\begin{remark}
    The analysis of numerical approximations of poroelastic models has received much attention in the last decade. A big focus has been developing solvers based on sequentially solving the mechanics and flow subproblems. The most popular iterative methods are the fixed-stress split and the undrained split, 
    based on adding a stabilization term to either the flow or the mechanics equation \cite{kim2011}. The convergence of these schemes was analyzed in \cite{bause2017,both2017,mikelic2013,storvik2019}. A parallel fixed-stress scheme, which allows for solving the mechanics problems in parallel for all time steps, was proposed in \cite{borregales2018}. The iterative schemes for the quasi-static Biot have been extended and analyzed for a dynamic Biot model for soft poroelastic materials in \cite{both2022}. A splitting scheme for the dynamic Biot model is proposed in \cite{bause2019}, with the proof being based on gradient flow techniques \cite{both2019}. Furthermore, the fixed-stress split has been extended to a history-dependent Biot model \cite{stokke2024}. 
\end{remark}
	
	\section{Solution theory for evolutionary problems}\label{sec:evo}
	
	In this section, we give a brief presentation of the unified solution theory for evolutionary equations and analyze the problem \eqref{eqs:biotallardADE} within this setting. We use standard notation for Sobolev spaces. First, we introduce some preliminary definitions. Let $H$ be a real Hilbert space equipped with the norm $\|\cdot\|_{H}$, then for a given parameter $\nu\in\R$ we define the weighted Bochner space and its associated inner product 
	\begin{align}\label{eq:space}
		H_{\nu}(\R;H):=\left\{f:\R\to H:\,\int_{\R}\|f(t)\|_{H}^{2}e^{-2\nu t}dt\right\},\quad
		\left\langle f,g\right\rangle_{H_{\nu}}:= \int_{\R}\left\langle f,g \right\rangle_{H}e^{-2\nu t}dt,
	\end{align}
 for $f,g\in H_{\nu}(\R;H)$.
	The inner product \eqref{eq:space} induces a norm  $\|\cdot\|_{\nu}$.
	Let $C^{\infty}_{c}(\R;H)$ be the space of  infinitely differentiable $H$-valued functions with compact support, then we denote by $\p_t$ the closure of the operator 
	\begin{align*}
		\p_t:C^{\infty}_{c}(\R;H)\subset H_{\nu}(\R;H)\to H_{\nu}(\R;H),\,\,\phi\to\phi'.
	\end{align*}
	
	\begin{definition}
		Let $\nu\in\R$ and $\mathcal{F}$ be the Fourier transform. We define the Fourier-Laplace transform as
  \vspace{-0.1cm}
		\begin{align*}
			\mathcal{L}_{\nu}:H_{\nu}(\R;H)\to H_{\nu}(\R;H),\, f\mapsto \mathcal{F}e^{-\nu t}f.
		\end{align*}
	\end{definition}
	
	\begin{lemma}[{\cite[Proposition 2.6]{trostorff2015}}]
		For $\nu\in\R_{>0}$ there holds that $\p_t=\mathcal{L}_{\nu}^{*}(\nu+\i t)\mathcal{L}_{\nu}$.
	\end{lemma}
	\begin{definition}\label{def:33}
		Let $M:D(M)\subset\mathbb{C}\to L(H)$ be a bounded, analytic function. Then, we define the linear operator 
		\vspace*{-0.5cm}
		\begin{align*}
			M(\nu+\i m):L^{2}(\R;H)\to L^{2}(\R,H)
		\end{align*}
		by $(M(\nu+\i m)f)(t)=M(\nu+\i t)f(t)$ for $t\in\R$ and the linear material law $M(\p_t)$ by
		\begin{align*}
			M(\p_t):=\mathcal{L}_{\nu}^{*}M(\nu+\i m)\mathcal{L}_{\nu}\in L(H_{\nu}(\R;H)).
		\end{align*}
	\end{definition}

	\begin{theorem}[{\cite[Theorem 6.2.1]{seifert2022}} ]\label{thm}
		Let $\nu_0\in\R$ and $H$ be a Hilbert space. Let $M:D(M)\subset\mathbb{C}\to L(H)$ be a selfadjoint material law and let $A:D(A)\subset H\to H$ be skew-selfadjoint. Assume that 
		\begin{equation}\label{solutioncondition}
			{\rm Re}\langle\phi,zM(z)\phi\rangle\geq c\|\phi\|^{2}, \quad  \phi\in H,\quad z\in\mathbb{C}_{{\rm Re}\geq \nu_0}
		\end{equation}
		for some $c>0$. Then for every $\nu\geq\nu_0$ the operator $\p_t M(\p_t)+A$ is closable in $H_{\nu}(\R; H)$ and for $\bm{G}\in H_{\nu}(\R;H)$ there exists a unique solution $\bm{U}\in H_{\nu}(\R;H)$ such that 
		\begin{align}\label{closureop}
			\widebar{(\p_t M(\p_t)+A)}\bm{U}=\bm{G}.
		\end{align}
		In addition, $\bm{U}$ satisfies the stability estimate $\|\bm{U}\|_{\nu}\leq c^{-1}\|\bm{G}\|_{\nu}$.
	\end{theorem}
	This means that under the hypotheses of Theorem \ref{thm}, the problem \eqref{eq:abstractproblem} is well-posed, i.e. the
	uniqueness, existence and continuous dependence on the data $\bm{G}$ of a solution $\bm{U}$ are ensured. We note that when $\bm{G}$ is in the domain of the time derivative of $s$-order, denoted by $\p_t^{s}$, being the space $H^{s}_{\nu}(\R; H)$, then $\bm{U}\in H^{s}_{\nu}(\R;H)$ and the evolutionary problem is solved literally
	\begin{align}\label{literal}
		(\p_t M(\p_t)+A)\bm{U}=\bm{G}.
	\end{align} Before writing \eqref{eqs:biotallardADE} as an evolutionary problem of the form \eqref{eq:abstractproblem}, we define the spatial differential operators.

	\begin{definition}
		For an open empty set $\Om\subset\R^d$, where $d\in\N$, we define
		\begin{align*}
			L^{2}(\Om)_{\rm sym}^{d\times d}:= \left\{(\phi_{ij})_{i,j=1,2,...,d}\in L^{2}(\Om): \phi_{ij}=\phi_{ji},\,\forall i,j\in\{1,...,d\}\right\}.
		\end{align*}
	\end{definition}
	\begin{definition}
		For an open empty set $\Om\subset\R^d$, where $d\in\N$, we set
		\begin{align}
			\tgrad_0:H_{0}^{1}(\Om)\subset L^{2}(\Om)\to L^{2}(\Om)^{d}, \,\,\phi\to(\p_j\phi)_{j=1,...,d}
		\end{align}
		and
  \vspace{-0.5cm}
		\begin{align}
			\tGrad_0:H_{0}^{1}(\Om)^{d}\subset L^{2}(\Om)^{d}\to 	L^{2}(\Om)_{\rm sym}^{d\times d}, \,\,(\phi)_{j}\to\frac12(\p_i\phi_j+\p_j\phi_i)_{i,j=1,...,d}.
		\end{align}
		In addition, we define
		\begin{align}\label{skew1}
			\tdiv: D(\tdiv)\subset L^{2}(\Om)^{d}\to L^{2}(\Om), \,\,\tdiv:=-(\tgrad_0)^{*}
		\end{align}
		and
  \vspace{-0.5cm}
		\begin{align}\label{skew2}
			\tDiv: D(\tDiv)\subset L^{2}(\Om)^{d\times d}_{\rm sym}\to L^{2}(\Om)^{d}, \,\,\tDiv:=-(\tGrad_0)^{*}.
		\end{align}
	\end{definition}
	We note that for any $\bmu\in H_{0}^{1}(\Om)^{d}$ there holds that $\e(\bmu)=\tGrad_0\bmu$.
	To rewrite \eqref{eqs:biotallardADE} as a first-order evolutionary problem, we introduce new unknowns
	$
	\bmv:=\p_t\bmu$ and $\bm{\sigma}:=\textbf{C}\e$.
	This results in the system
	\begin{subequations}\label{eqs:firstorderADE}
		\begin{align}
			\rho\p_t\bmv -\mbox{Div}\bm{\sigma} +\a \mbox{grad}_0 p+\rho_f\p_t \Psi_{1}&=\rho\bmF,\\
			\bm{S}\p_t\bm{\sigma}- \tGrad_0 \bmv &=0,\\
				c_0\p_t p +\a \mbox{div} \bm{v}+\mbox{div}\Psi_1&=0,\\
			c_1\rho_f Fd_1^{-1}\eta_k^{-1}\p_t\Psi_1 +\rho_f Fd_1^{-1}\eta_k^{-1}\Psi_1+\tgrad_{0} p+\rho_f\p_t\bmv&=\rho_f\bmF,
		\end{align}
	\end{subequations}
		where $\bm{S}$ denotes the positive definite, fourth-order compliance tensor of the inverse stress-strain relation of Hook's law of linear elasticity $\e=\bm{S} \bm{\sigma}$.
	By denoting
	\begin{align}\label{data}
		\bm{U}:=(\bmv,\bms,p,\Psi_1)^{T}\mbox{ and } \bm{G} :=(\rho\bmF,0,0,\rho_f\bmF)^{T},
	\end{align}
	and defining the following operators
	\begin{align}\label{operators}
		M_0:=\begin{pmatrix}
			\rho & 0 & 0 & \rho_f\\
			0 & \bm{S} & 0 & 0\\
			0 & 0 & c_0 & 0\\
			\rho_f& 0& 0&\frac{c_1\rho_f F}{d_1\eta_k}
		\end{pmatrix},\,M_1:=\begin{pmatrix}
			0 & 0 & 0 & 0\\
			0 & 0 & 0 & 0\\
			0 & 0 & 0 & 0\\
			0 & 0 & 0 & \frac{\rho_f F}{d_1\eta_k}\\
		\end{pmatrix}, \,A:=\begin{pmatrix}
			0 & -\mbox{Div} & \a \mbox{grad}_0& 0\\
			-\mbox{Grad}_0 & 0 & 0 & 0\\
			\a\mbox{div} & 0 & 0 & \mbox{div}\\
			0 & 0 & \mbox{grad}_0 & 0\\
		\end{pmatrix},
	\end{align}
	the equations \eqref{eqs:firstorderADE} can be written as the following evolutionary problem.
	\begin{problem}[Evolutionary problem]\label{problem}
		Let $\bm{H}$ denote the product space
		\begin{equation}
			\bm{H}:=L^{2}(\Om)^{d}\times L^{2}(\Om)^{d\times d}_{\rm sym}\times L^{2}(\Om)\times L^{2}(\Om)^{d},
		\end{equation}
		endowed with the $L^{2}$ inner product of $L^{2}(\Om)^{(d+1)^{2}}$. Let $M_0,M_1:\bm{H}\to \bm{H}$ and $A:D(A)\subset \bm{H}\to \bm{H}$, with
		\begin{align}
			D(A): = H_0^{1}(\Om)^{d}\times D(\tDiv)\times H_{0}^{1}(\Om)\times D(\tdiv),
		\end{align}
		be defined by \eqref{operators}. For given $\bm{G}$ by \eqref{data}, find $U\in H_{\nu}(\R;\bm{H})$ such that 
		\begin{equation}\label{eq:22}
			(\p_tM_0+M_1+A)\bm{U}=\bm{G}.
		\end{equation} 
	\end{problem}

\begin{theorem}\label{thm:new}
 Consider \Cref{problem} with parameters of \eqref{relation:dynamicpermeability}. Let $\nu_0\in\R_{>0}$ and assume that the inequality $\left(\nu_0\left(\frac{c_1 F}{d_1\eta}-\frac{1}{\rho}\right)+\frac{ F}{d_1\eta}\right)>0$ holds. Then, the material law associated with \Cref{problem} satisfies \eqref{solutioncondition}, and hence \Cref{problem} is well-posed, and there exists a unique solution of \eqref{eq:22} in the sense of \eqref{closureop}.
		
	\end{theorem}
 
		
	\begin{proof}
	  The operators $M_0,M_1$ in \eqref{operators} correspond to a material law in the form
	\begin{align}\label{eq:26}
		M(z)=\begin{pmatrix}
			\rho & 0 & 0 & \rho_f\\
			0 & \bm{S} & 0 & 0\\
			0 & 0 & c_0 & 0\\
			\rho_f& 0& 0&\frac{c_1\rho_f F}{d_1\eta_k}
		\end{pmatrix}+\frac{1}{z}\begin{pmatrix}
			0 & 0 & 0 & 0\\
			0 & 0 & 0 & 0\\
			0 & 0 & 0 & 0\\
			0 & 0 & 0 & \frac{\rho_f F}{d_1\eta_k}\\
		\end{pmatrix}.
	\end{align}
	By using symmetric Gauss steps as congruence transformation, the material law operator is congruent to 
	\begin{align}
	\begin{pmatrix}
		\rho & 0 & 0 & 0\\
		0 & \bm{S} & 0 & 0\\
		0 & 0 & c_0 & 0\\
		0& 0& 0&\frac{c_1\rho_f F}{d_1\eta_k}-\frac{\rho_f^{2}}{\rho}
	\end{pmatrix}+\frac{1}{z}\begin{pmatrix}
		0 & 0 & 0 & 0\\
		0 & 0 & 0 & 0\\
		0 & 0 & 0 & 0\\
		0 & 0 & 0 & \frac{\rho_f F}{d_1\eta_k}\\
	\end{pmatrix}.
	\end{align}
	Let $z=\nu+\i m\in\mathbb{C}_{\rm Re\geq\nu_0}$ and $\phi=(\bmv,\bm{\sigma},p,\Psi_1)\in \bm{H}$. The compliance tensor $\bm{S}$ is symmetric and uniformly positive definite in the sense that ${\rm Re} \bm{S}\geq c_s>0$. Then by the positivity of the constants it follows that
	\begin{align}
		{\rm Re}\langle \phi,zM(z)\phi\rangle_{\scriptscriptstyle H}&=\nu\langle \rho \bmv,\bmv\rangle_{ \scriptscriptstyle L^{2}(\Om)^{d}}+{\rm Re}\langle \bms, zS\bms\rangle_{\scriptscriptstyle L^{2}(\Om)^{d\times d}_{\rm sym}}+\nu\langle c_0p,p\rangle_{\scriptscriptstyle L^{2}(\Om)}\nonumber\\&\,\,+\left(\nu\left(\frac{c_1\rho_f F}{d_1\eta_k}-\frac{\rho_f^{2}}{\rho}\right)+\frac{\rho_f F}{d_1\eta_k}\right)\left\langle \Psi_1,\Psi_1\right\rangle_{\scriptscriptstyle L^{2}(\Om)^{d}}\nonumber\\
		&\geq \nu \rho\| \bmv\|^{2}_{ \scriptscriptstyle L^{2}(\Om)^{d}}+\nu c_s\|\bms\|^{2}_{ \scriptscriptstyle L^{2}(\Om)^{d\times d}_{\rm sym}}+\nu c_0\| p\|^{2}_{ \scriptscriptstyle L^{2}(\Om)}+\left(\nu\left(\frac{c_1\rho_f F}{d_1\eta_k}-\frac{\rho_f^{2}}{\rho}\right)+\frac{\rho_f F}{d_1\eta_k}\right)\| \Psi_1\|^{2}_{ \scriptscriptstyle L^{2}(\Om)^{d}}\nonumber\\
		&\geq \min\left\{\nu_0 \rho,\nu_0 c_s,\nu_0 c_0,\rho_f^{2}\left(\nu_0\left(\frac{c_1 F}{d_1\eta}-\frac{1}{\rho}\right)+\frac{F}{d_1\eta}\right)\right\}\|\phi\|_{\scriptscriptstyle H}^{2}= c_{\min}\|\phi\|_{\scriptscriptstyle H}^{2}.\label{eq:25}
	\end{align}
	By the assumption of the theorem, the positivity of $c_{\min}$ is ensured. Therefore, by \eqref{eq:25} the material law \eqref{eq:26} satisfies \eqref{solutioncondition}. Also, from \eqref{skew1} and \eqref{skew2} it follows that $A$ is skew-selfadjoint and then by Theorem~\ref{thm} \Cref{problem} is well-posed.
	\end{proof}
\begin{remark}\label{remark: assumption}
    The inequality assumed in Theorem \ref{thm:new} holds for a large class of applications. First, we note that $F\geq 1$ and that it becomes large for materials with low porosity. Secondly, in many cases, we have that $\rho>1$ and $\eta<1$. Experiments for materials with higher porosity, like cancellous bone, indicate that $c_1>d_1$ \cite{yvonne2014}. By considering the material values for an epoxy-glass mixture and sandstone saturated with water from \cite{yvonne2019}, one also sees that $c_1>d_1$. The computation of the constants for $N=4$ and different material parameters can be redone at \url{https://github.com/MrShuffle/DynamicPermeabilityConstants/releases/tag/v1.2}
    \end{remark} 
	\begin{remark}
		The existence and uniqueness of a solution to the system \eqref{eqs:BiotAllardconv} in the convolution form can also be shown using \Cref{thm}. In order to show the solvability condition \eqref{solutioncondition} for the corresponding material law we need to assume that $\calA\in L_{\nu}^{1}(\R)$ which implies that the operator $\calA\ast$ is linear and bounded \cite[Lemma 3.1]{trostorff2015}. In addition, for $\nu_0\in\R_{>0}$ the assumption that ${\rm Re}(\sqrt{2\pi}\hat{\calA}(-\i(\i m+\nu_0)))^{-1}-\nu_0\frac{\rho_f}{\rho}>0$  holds is necessary to prove the inequality \eqref{solutioncondition}. This assumption is satisfied for a large class of problems if $\nu_0$ is chosen sufficiently large. 
	\end{remark}

	\section{General approximation of the dynamic permeability with $N>1$ in Eq.\ \eqref{relation:dynamicpermeability}}\label{sec: general}
	In the general case, when $N>1$ in the approximation of the dynamic permeability \eqref{relation:dynamicpermeability} in the frequency domain, we get an extended Biot--Allard system by introducing the auxiliary variables $\hat{\Psi}_j$, defined by
	\begin{align}\label{eq:auxj}
		\hat{\Psi}_j := \frac{\eta_k}{F}\frac{d_j}{1+\i\om c_j}\left(\hat{\bmF}-\frac{1}{\rho_f}\del \hat{p}+\om^{2}\hat{\bmu}\right), \quad\forall j=1,...,N.
	\end{align}
	Then, by taking the Fourier inverse of \eqref{eq:auxj} we obtain the augmented Biot--Allard system
	\begin{subequations}\label{eqs:biotallardADEgeneral}
		\begin{align}
			\rho\p_{t}^2\bmu-\div\bm{C}\e(\bmu)+\a\del p +\rho_f\sum_{j=1}^{N}\p_t \Psi_j&=\rho\bmF,\\
			c_0\p_tp+\div (\alpha\p_t\bmu)+\sum_{j=1}^{N}\div \Psi_j&=0,\\
			\frac{c_j\rho_f F}{d_j\eta_k}\p_t\Psi_j+\frac{\rho_f F}{d_j\eta_k}\Psi_j +\del p+\rho_f\p_t^{2}\bmu&=\rho_f\bmF, \quad\forall j=1,...,N.
		\end{align}
	\end{subequations}
	The well-posedness of \eqref{eqs:biotallardADEgeneral} written as an evolutionary problem can be obtained similarly to the case of \Cref{thm} when $N=1$, since the skew-adjointness of the operator $A$ is preserved. Thus a unique solution to \eqref{eqs:biotallardADEgeneral} rewritten as an evolutionary problem exists under the assumption that the inequality
 \begin{align}\label{ass:general}
 \left(\nu_0\left(\frac{c_j F}{d_j\eta}-\frac{1}{\rho}\right)+\frac{F}{d_j\eta}\right)>0, \quad\forall j=1,...,N,
 \end{align}
 hold for some $\nu_0\in\R_{>0}$.
 For the constants $c_j,d_j$ in \eqref{ass:general}, \Cref{remark: assumption} continues to hold. 

\section{Conclusions}
  We considered the Biot--Allard model for flow in deformable porous media with memory effects. We derived
	a new equivalent formulation where the memory effects are incorporated by auxiliary differential equations. We proved its well-posedness by an abstract solution theory for evolutionary problems. 
  The advantage of writing the Biot--Allard model in the convolution-free form \eqref{eqs:biotallardADEgeneral} becomes evident when we turn towards its numerical discretization. The expensive computation of time convolution integrals is avoided. This allows for fast and efficient numerical schemes. For the feasibility and efficiency of the ADE approach \eqref{aux1}-\eqref{cor} in non-linear optics, we refer to \cite{margenberg2024}. The approximation of the dynamic Biot \eqref{eqs:BiotAllardconv} without memory terms by families of space-time finite element methods, based on a uniform variational formulation in space and time, is proposed and investigated in \cite{bause2024}. Structure-preserving space-time finite element methods for the dynamic Biot system are considered in \cite{kraus2024}. 

 

	\section*{Acknowledgments}
	JSS and FAR acknowledge the support of the VISTA program, The Norwegian
	Academy of Science and Letters, and Equinor. We wish to thank M. Jakobsen for helpful discussions during this work.

	


\begin{thebibliography}{00}

\bibitem{batzle2001} Batzle, M., Hofmann, R., Han, D.-H., Castagna, J.: Fluids and frequency dependent seismic velocity of rocks. Lead.\ Edge {\bf 20}(2), 168–171 (2001).
		
		\vspace{-0.3cm}
		
		\bibitem{bause2024}	Bause, M., Anselmann, M., Köcher, U., Radu, F.A.: Convergence of a continuous Galerkin method for hyperbolic-parabolic systems. Comput.\ Math.\ with Appl. {\bf 158}, 118–138 (2024). 

        \vspace{-0.3cm}
        \bibitem{bause2019} Bause, M., Both, J.W., Radu, F.A.: Iterative Coupling for Fully Dynamic Poroelasticity. In:
        Numerical Mathematics and Advanced Applications ENUMATH 2019. pp. 115–123. Springer Cham (2021).
        \vspace{-0.3cm}
        \bibitem{bause2017} Bause, M., Radu, F.A., Köcher, U.: Space–time finite element approximation of the Biot poroelasticity system with iterative coupling. Comput.\ Methods Appl.\ Mech.\ Eng. {\bf 320}, 745–768 (2017). 
        
        \vspace{-0.3cm}
        
        \bibitem{borregales2018} Borregales, M., Kumar, K., Radu, F.A., Rodrigo, C., Gaspar, F.J.: A partially parallel-in-time fixed-stress splitting method for Biot’s consolidation model. Comput.\
        Math.\ Appl. {\bf 77}, 1466–1478 (2019). 



		\vspace{-0.3cm}

        \bibitem{both2022} Both, J.W., Barnafi, N.A., Radu, F.A., Zunino, P., Quarteroni, A.: Iterative splitting schemes for a soft material poromechanics model. Comput.\ Methods Appl.\ Mech.\ Eng. {\bf 388}, 114183 (2022). 
        \vspace{-0.3cm}
        \bibitem{both2019} Both, J.W., Kumar, K., Nordbotten, J.M., Radu, F.A.: The gradient flow structures of thermo-poro-visco-elastic processes in porous media, arXiv preprint arXiv:1907.03134, (2019)


        \vspace{-0.3cm}
        
        \bibitem{both2017} Both, J.W., Borregales, M., Nordbotten, J.M., Kumar, K., Radu, F.A.: Robust fixed stress splitting for Biot’s equations in heterogeneous media. Appl.\ Math.\ Lett. {\bf 68}, 101–108 (2017). 


        \vspace{-0.3cm}
  
		\bibitem{grafakos2014} Grafakos, L.: Classical Fourier Analysis. Grad.\ Texts in Math. {\bf 249}, Springer (2014).



        \vspace{-0.3cm}
        
		\bibitem{johnson1987} Johnson, D.L., Koplik, J., Dashen, R.: Theory of dynamic permeability and tortuosity in fluid-saturated porous media. J.\ Fluid Mech. {\bf 176}, 379–402 (1987).
		
		\vspace{-0.3cm}
		\bibitem{kim2011} Kim, J., Tchelepi, H.A., Juanes, R.: Stability and convergence of sequential methods for coupled flow and geomechanics: Drained and undrained splits. Comput.\ Methods Appl.\ Mech.\ Eng. {\bf 200}, 2094–2116 (2011). 

        \vspace{-0.3cm}
        
		\bibitem{kraus2024} Kraus, J., Lymbery, M., Osthues, K., Philo, F.: Analysis of a family of time-continuous strongly conservative space-time finite element methods for the dynamic Biot model, arXiv preprint arXiv:2401.04609, (2024).
		
		\vspace{-0.3cm}

        \bibitem{margenberg2024}  Margenberg, N., Kärtner, F.X., Bause, M.: Optimal Dirichlet boundary control by Fourier neural operators applied to nonlinear optics. J. Comp. Phys. {\bf 499}, 112725 (2024). 

        \vspace{-0.3cm}

        \bibitem{mikelic2013} Mikelić, A., Wheeler, M.F.: Convergence of iterative coupling for coupled flow and geomechanics. Comput.\ Geosci. {\bf 17}, 455–461 (2013). 


        
        \vspace{-0.3cm}
    
		\bibitem{MikelicBiotAllard}  Mikelić, A., Wheeler, M.F.: Theory of the dynamic Biot-Allard equations and their link to the quasi-static Biot system. J. Math. Phys. {\bf 53}(12), 123702 (2012).
		\vspace{-0.3cm}
		\bibitem{picard2009} Picard, R.: A structural observation for linear material laws in classical mathematical physics. Math.\ Meth.\ in the Appl.\ Sci. {\bf 32}, 1768–1803 (2009).

        \vspace{-0.3cm}
		\bibitem{robling2006} Robling, A.G., Castillo, A.B., Turner, C.H.: Biomechanical and molecular regulation of bone remodeling. Annu.\ Rev.\ Biomed.\ Eng. {\bf 8}(1), 455–498 (2006).
		\vspace{-0.3cm}
		\bibitem{seifert2022} Seifert, C., Trostorff, S., Waurick, M.: Evolutionary Equations: Picard’s Theorem for Partial Differential Equations, and Applications. Springer Cham (2022).
		\vspace{-0.3cm}
		\bibitem{showalter2000} Showalter, R.E.: Diffusion in Poro-Elastic Media. J.\ Math.\ Anal.\ Appl. {\bf 251}(1), 310–340 (2000).

  \vspace{-0.3cm}
  \bibitem{stokke2024} Stokke, J.S., Jakobsen, M., Kumar, K., Radu, F.A.: A History-dependent Dynamic Biot Model, arXiv preprint arXiv:2402.07669, (2024).
  \vspace{-0.3cm}
  \bibitem{storvik2019} Storvik, E., Both, J.W., Kumar, K., Nordbotten, J.M., Radu, F.A.: On the optimization of the
fixed-stress splitting for Biot’s equations. Int J Numer Methods Eng. {\bf 120}(2), 179–194 (2019)




		\vspace{-0.3cm}
		\bibitem{trostorff2015}  Trostorff, S.: On integro‐differential inclusions with operator‐valued kernels. Math.\ Meth.\ in the Appl.\ Sci.  {\bf 38}, 834–850 (2015).
		\vspace{-0.3cm}
		\bibitem{yvonne2019} Yvonne Ou, M.-J., Woerdeman, H.J.: On the Augmented Biot-JKD Equations with Pole-Residue Representation of the Dynamic Tortuosity. In Bolotnikov, V. et al. (eds.): Interpolation and Realization Theory with Applications to Control Theory. Birkhäuser Cham, 307–328 (2019).


		\vspace{-0.3cm}
		\bibitem{yvonne2014} 	Yvonne Ou, M.-J.: On reconstruction of dynamic permeability and tortuosity from data at distinct frequencies. Inverse Probl. {\bf 30}(9), 095002 (2014).
		

\end{thebibliography}
	


\end{document}